\documentclass[12pt]{article}
\usepackage{amssymb,amsmath,amsfonts,amsthm,amsxtra,setspace}
\usepackage[T1]{fontenc}
\usepackage{lmodern}
\usepackage{titlesec}
\titlelabel{\thetitle.\,}

\newcommand{\noopsort}[1]{}

\newtheorem{theorem}{Theorem}[section]
\newtheorem{cor}[theorem]{Corollary}

\newsavebox{\Prfref}

\newsavebox{\prfref}

\newtheoremstyle{ref}% an optional reference is inserted after the theorem number
{\topsep}	%      Space above
{\topsep}	%      Space below
{\it}%         		Body font
{}%         		Indent amount (empty = no indent, \parindent = para indent)
{}%			Thm head font
{}%        		Punctuation after thm head
{ }%			Space after thm head: " " = normal interword space;
{\thmname{{\bfseries#1}}\thmnumber{ \textbf{#2\thmnote{\rm #3}\textbf .}}}%Thm head spec (can be left empty, meaning `normal')

\theoremstyle{ref}
\newtheorem{lem}[theorem]{Lemma}
\newtheorem{thm}[theorem]{Theorem}

\newtheorem{defn}[theorem]{Definition}

\newtheoremstyle{nnref}% an optional reference is inserted after no theorem number
{\topsep}	%      Space above
{\topsep}	%      Space below
{}%         		Body font
{}%         		Indent amount (empty = no indent, \parindent = para indent)
{}%			Thm head font
{}%        		Punctuation after thm head
{ }%			Space after thm head: " " = normal interword space;
{\thmname{\textbf{#1}\thmnote{\textrm{ #3}}\textbf{.}}}%Thm head spec (can be left empty, meaning `normal')

\theoremstyle{nnref}

\begin{document}

\title{ The Open Graph Axiom and Menger's Conjecture}
\author{Franklin D. Tall{$^1$}, Stevo Todorcevic{$^2$}, Se{\c{c}}il Tokg\"oz{$^3$}}

\footnotetext[1]{Research supported by NSERC grant A-7354.\vspace*{2pt}}
\footnotetext[2]{Research supported by grants from CNRS and NSERC 455916.}
\footnotetext[3]{ Research supported by T\"UB\.{I}TAK  grant 2219.\vspace*{2pt}}
\date{\today}
\maketitle

\begin{abstract}
\noindent Menger conjectured that subsets of $\mathbb R$ with the Menger property must be $\sigma$-compact. While this is false when there is no restriction on the subsets of $\mathbb R$,  for projective subsets it is known to follow from the Axiom of  Projective Determinacy, which has considerable large cardinal consistency strength. We show that  the perfect set version of the Open Graph Axiom for projective sets of reals, with consistency strength only an inaccessible cardinal, also implies Menger's conjecture restricted to this family of subsets of $\mathbb R$.
\end{abstract}

\renewcommand{\thefootnote}{}
\footnote
{\parbox[1.8em]{\linewidth}{  2010 MSC. Primary 03E15, 03E35, 03E60, 54A25, 54D20, 54H05; Secondary 03E45.}}\\
\renewcommand{\thefootnote}{}
\footnote
{\parbox[1.8em]{\linewidth}{Keywords and phrases: Menger, $\sigma$-compact, projective set of reals, Hurewicz Dichotomy,
 Open Graph Axiom.}}

\section{Introduction}
In $1924$,  Menger \cite{Menger} introduced  a topological property for metric spaces which he referred to as  \textbf{ ``property E''}.  Hurewicz \cite{Hur27} reformulated the property  E as the following, nowadays called the \textbf{Menger property}:

\begin{defn} A space $X$ is Menger if whenever $\{{\mathcal U}_{n}\}_{n\in\mathbb{N} }$ is a sequence of open covers, there exist finite ${{\mathcal V}}_{n}\subseteq{\mathcal U}_{n}$, $n\in\mathbb{N}$, such that $\bigcup_{n\in\mathbb{N}} {\mathcal V}_{n}$ is a cover of $X$.
\end{defn}

There has recently been interest in the question of whether ``definable'' Menger spaces {\textemdash } and, more specifically, Menger sets of reals {\textemdash }  are $\sigma$-compact. See e.g., \cite{Tallnew,TTo,Tok}. Hurewicz \cite{Hur25} refuted Menger's conjecture \cite{Menger} under CH. Just et al \cite{COC2}  refuted the Hurewicz conjecture (and hence  the Menger conjecture) in ZFC. Also, a ZFC counterexample was produced by Chaber and Pol \cite{CP} in an unpublished note.  More natural examples were produced by Bartoszy\' nski and Shelah \cite{BaSh01}, and later Tsaban and Zdomskyy \cite{TsZ}. Hurewicz \cite{Hur25} proved that analytic Menger subsets of  $\mathbb R$ are $\sigma$-compact; this was later extended to arbitrary  Menger analytic spaces by Arhangel'ski{\u\i} \cite{A}. Hurewicz  \cite{Hur25} also proved this for completely metrizable spaces; this was extended to \v{C}ech-complete spaces in \cite{TTo}. That determinacy hypotheses suffice was first noticed in \cite{MF} and stated explicitly in \cite{Tall2011}. See also \cite{TTs} and \cite{BT}.\\

In this article, we will investigate  the influence of an appropriate version of the Open Graph Axiom on Menger's conjecture. To introduce this axiom, let $X$ be  a separable metrizable space. We denote the family of all unordered  pairs of its elements by $[X]^{2}=\{(x,y): x,y\in X$ and $x\neq y\}$.   A {\bf graph} on $X$ is a structure of the form $G=(X, E)$ where $X$ is a set of {\bf vertices}, and  $E\subseteq [X]^2$ is a symmetric {\bf edge set}. In other words, $E$ is a symmetric and irreflexive binary relation on $X$ and the pairs in $E$ are called {\bf{edges}} of $G$.  If $E$ is open  in the product space $X\times X$, then $G=(X,E)$ is called an {\bf open graph}.  An open graph of the form $G=(X,E)$ is {\bf countably chromatic} if there is a  decomposition  $X=\bigcup_{n\in\mathbb N} X_n$  such that $[X_n]^2\cap E=\emptyset$  for all $n\in\mathbb N$.  Recall that a {\bf complete subgraph} of $X$ is a subset $Y$ of $X$ such that $[Y]^2\subseteq E.$   Note that countably chromatic graphs do not have uncountable complete subgraphs. So the following dichotomy is one of the natural statements that happen to be consistent with ZFC (see \cite{T1} and \cite{T3}).\\

\textbf{The Open Graph Axiom (OGA).} If $G=(X,E)$ is an open graph on a separable metric space $X$, then $G$ is either countably chromatic or it includes an uncountable  complete subgraph. \\

Connections between \textbf{OGA} and Menger's conjecture are implicit in \cite{T2}. We shall explore those further here.\\

Sometimes one views dichotomies of the form \textbf{OGA} as Ramsey-theoretic statements or coloring axioms.
Recall that a version of an open coloring axiom for sets of reals of cardinality $\aleph_1$ was introduced by  Avraham, Rubin, and Shelah \cite{ARS}. Another version of an open coloring axiom for arbitrary sets of reals was introduced by  the second author  \cite{T1}. We will use \textbf{OGA} for the later version as it is clearly equivalent to \textbf{OGA} stated above.\\
 
 \section{ $\mathbf{OGA^*(\Gamma)}$ and the Hurewicz Dichotomy}
 
  A classical phenomenon,  the \emph{Hurewicz Dichotomy}, was first investigated by Hurewicz \cite{Hur28} and later extended by Kechris, Louveau and Woodin \cite{KLW}. See e.\,g. Section $21.F$ of  \cite{Kech}. Here is one version of the Hurewicz Dichotomy.
\\

\textbf{Hurewicz Dichotomy (HD).}  Let $A\subseteq\mathbb{R}$ be analytic. If $A$ is not $\sigma$-compact, then there is a Cantor set $K\subseteq \mathbb{R}$ such that $K\cap A$ is dense in $A$  and homeomorphic to  $\mathbb{P}$, the space of irrationals, and $K\setminus A$ is  countable dense  in $K$ and homeomorphic to $\mathbb{Q}$, the space of rationals. 

%Let $E$ be  a compact metrizable space, and let $A$ and $B$ be two disjoint subsets of $E$. If there is no $K{_\sigma}$-set, i.e., countable union of compact subsets, $C$ such that $A\subseteq C$ and $C\cap B=\emptyset$, then there is a Cantor set $K$ such that $K\subseteq A\cup B$  and $ K \cap B$ is countable dense in $K$.

\begin{defn} Let $\Gamma$ be a subset of the power set of $\mathbb{R}$. $\mathrm{\mathbf{HD(\Gamma)}}$ is the assertion obtained from $\mathbf{HD}$ by replacing ``analytic''  by `` a member of $\Gamma$'' .
\end{defn}
%\begin{thm}[~ \cite{Hur28}]  Let $X$ be a Polish space (a separable, completely metrizable space). Then $X$ includes a closed copy of $\mathbb{P}$ (the space of irrationals)  if and only if $X$ is not $\sigma$-compact.
%\end{thm}

%It is also of interest that:
%\begin{thm}[~ \cite{KLW}]  Assume  the Axiom of Determinacy. Any subset of  a compact metrizable space satisfies HD.
%\end{thm}

%Weaker forms of determinacy not contradicting the Axiom of Choice yield smaller classes of subsets satisfying \textbf{HD} \cite{Tall2011}.

\begin{thm}[~\cite{Hur28}] If  $\Gamma$ is a collection of subsets of $\mathbb{R}$ satisfying $\mathrm{\mathbf{HD(\Gamma)}}$ as above  then every Menger member of $\Gamma$ is $\sigma$-compact.

\end{thm}
\begin{proof} Let $A$ be a Menger member of $\Gamma$. Suppose $A$ is not  $\sigma$-compact.  By \textbf{HD($\Gamma$)},  there is a Cantor set $K$ such that $K\subseteq \mathbb{R}$  and $ K\setminus A$ is countable dense in $K$. Then 
$K\cap A$ is homeomorphic to $\mathbb{P}$ {~\cite[p.\,160]{Kech}}. But $K\cap A$ is a closed subset of $A$ and  $\mathbb{P}$ is not Menger \cite{Hur27}; since that property is closed-hereditary,  $A$ cannot be Menger.
\end{proof}

Now  we recall  some basic notation and constructions:   A  partially ordered set $(T,\prec)$ is called a {\textbf{ tree}}  if for each $x\in T$, the set  $\{y: y\prec x \}$ is well-ordered by $\prec$. A {\textbf{ branch}} in $T$ is a maximal linearly ordered subset  of $T$. The {\textbf{ body of T}}, written as $[T]$, is the set of all infinite branches of $T$.  A tree  $T$ is called {\textbf{ pruned }} if every $s\in T$ has a proper extension $s\prec t, t\in T$. $[\mathbb{N}]^{< \mathbb{N}}$ denotes the collection of all finite subsets of $\mathbb{N}$, and it can be considered as a tree under end-extension. See  Section 2.C in \cite{Kech} for more details.\\

 Define $[{\mathbb{N}]^{<\mathbb{N}}}\otimes [{\mathbb{N}]^{< \mathbb{N}}}=\{(s,t)\in  [{\mathbb{N}]^{< \mathbb{N} }}\times  [{\mathbb{N}]^{<\mathbb{N}  }}:$ $\vert{s}\vert{=}\vert{t}\vert \}$. Then $[{\mathbb{N}]^{<\mathbb{N} }}\otimes [{\mathbb{N]}^{<\mathbb{N} }}$ is   also a tree with the product ordering $\sqsubseteq$, defined by  $(s,t)\sqsubseteq (s',t')$ if and only if  $s\subseteq s'$ and $t\subseteq t'$. \\

 For  $a,b\in {\mathcal P}(\mathbb{N)}$,  $a\subseteq ^{* } b$ ( $a$ is {\it almost included} in $b$) if $a\setminus b\in[\mathbb{N}]^{< \mathbb{N}} $.

 \begin{defn} 
\item 1. For $a,b\in{\mathcal P}(\mathbb{N})$, \textbf{ $a\perp b$} ($a$ and $b$ are orthogonal) if and only if \\  $a\ \cap \ b\in[\mathbb{N}]^{< \mathbb{N}}$.
\item 2. Let $\mathcal{A}$ and $\mathcal{B}$ be families of  subsets of $\mathbb{N}$. $\mathcal{A}\perp\mathcal{B}$ if and only if\\ $( \forall a\in \mathcal{A})( \forall b\in \mathcal{B})$ $a\perp b$.
\item 3. For a family  $\mathcal{A}$ of  subsets of $\mathbb{N}$, $\mathcal{A}^{\perp}=\{ b\in\mathcal P(\mathbb{N}): (\forall a\in \mathcal{A})$ $a\perp b \}$.
\item 4. For $\mathcal{A}\perp\mathcal{B}$, $\mathcal A$ is countably generated in $ \mathcal{B}^\perp$ if there is  a sequence $\{c_n\}$ of elements of $ \mathcal{B}^\perp$ such that every element of $\mathcal A$ is almost  included in one of  the $c_n$'s. 
 \end{defn}

Feng \cite{Fe} discussed the following natural  perfect set variation of $\mathbf{OGA}$ restricted to a given collection of subsets of $\mathbb R.$\\

\begin{defn} $\mathbf{OGA^*(\Gamma):}$  Let $\Gamma\subseteq \mathcal{P}(\mathbb R)$. Then for each open graph of the form $G=(X, E)$ where the vertex set  $X$ belongs to $ \Gamma$, either $G$ is countably chromatic or it includes a {\bf perfect} complete subgraph, i.e., a perfect set $P\subseteq X$ so that the induced subgraph is complete. \\
\end{defn}

\begin{thm}
Assume $\mathbf{OGA^*( \Gamma)}$  and that $\Gamma$  is closed under continuous pre-images. Then $\mathrm{\mathbf{HD(\Gamma)}}$  holds and  hence Menger $\Gamma$ sets are $\sigma$-compact.
\end{thm}

\begin{proof} Our argument is  similar to the proof of Theorem 3 in \cite{T2}.  
  Let $A\in  \Gamma$. Without loss of generality, let $A \subseteq (0,1)$. Suppose $A$ is not $\sigma$-compact.  
  Let $bA$ be the closure of $A$ in $[0,1]$. Define $B=bA\setminus A$. $bA$ is a continuous image of the Cantor space $ E=\{0,1\}^\mathbb{N}$ {~\cite[p.\,23]{Kech}}. Let $\pi: E\rightarrow bA$ be a continuous surjection, and set $A^{*}=\pi^{-1}(A)$,  $B^{*}=\pi^{-1}(B)$. Then $A^{*}\cap B^{*}=\emptyset$.  Moreover,  $A^*$ cannot be  $\sigma$-compact, since $A$ is not $\sigma$-compact. Therefore, if we could prove the result for  $A^*\in \Gamma$ and $A^*\subseteq E$, then we would have a Cantor set $E'\subseteq E$ with $E'\subseteq   A^{*}\cup B^{*} $, $E'\cap B^*$ countable dense in $E'$, $E'\cap A^*$ also dense in $E'$. Since $\pi$ is a surjection, $\pi(E') \cap B$ and $\pi(E') \cap A$ would be dense in $\pi(E')$. Clearly $K=\pi(E')$ is compact. $K$ has no isolated points, since it has two disjoint dense subsets. By  the Baire Category Theorem, $K$ must be uncountable. Since it is perfect, by a standard argument  it includes a  Cantor set $C$ such that   $C\subseteq A\cup B$  and $C\cap B$  is countable and  dense in $C${~\cite[p.\,162]{Kech}}.   
  
  To see this, one can construct a Cantor scheme $\{C_s\}_{s\in 2^{<{\mathbb{N}}}}$ by using open sets $C_s$ in $K$ with  $x_{s}\in C_{s}\cap B$ such that $diam(C_{s})<2^{-length(s)}$, $\overline{C_{s\frown i}}\subseteq C_s$, and $x_{s\frown 0}=x_s$ for all $s\in \{0,1\}^{{<{\mathbb{N}}}}$, $i\in\{0,1\}$. Then such a Cantor scheme determines a Cantor set $C=\bigcup\limits_{x\in E}\bigcap\limits_{n\in \mathbb{N}} C_{x\mid n}$ in $K$ such that $C\cap B$ is countable dense in $C${~\cite[p.\,162]{Kech}}. 
  
  To see $C\cap B$ is dense in $C$, take a nonempty open set $U$ in  $C$. Then there is a $x_{\tilde{s}} \in U$. By construction, $\{x_{\tilde{s}}\}=\bigcap\limits_{n\in \mathbb{N}} C_{x\mid n}$ for some $x\in E$. Since all $C_{x\mid n}$'s are open in $K$ and  $K\cap B$ is dense in $K$,  by construction of the points of $B$ in  the Cantor scheme,  $x_{\tilde{s}}\in(C_{x\mid n}\cap B)$ , and so $x_{\tilde{s}} \in U\cap(C\cap B)$.

Thus, without loss of generality, we may assume $A$ is a subset of  the Cantor space $E$.

 %By $exp(E)$  we denote the set of all closed (compact) subsets of $E$ with {\bf {the exponential topology}}, which is generated by the sets of the form $\{K\in exp(E): K\subseteq U\}$ or  $\{K\in exp(E): K\cap U\neq\emptyset\}$, where $U$  is open in $E$. Let us note that $exp(E)$ is homeomorphic to $E$ {~\cite[p.\,39]{T1}}. 
 Define  $\hat{A}= \bigcup\limits_{a\in A}\hat{a}$ where $\hat{a}$ is the set of all infinite chains of $\{0,1\}^{< \mathbb{N}}$ whose union is equal to $a$.  
Notice that a Cantor set can be viewed as the branches of the complete binary tree $\{0,1\}^{<\mathbb{N}}$. It follows that for any $a\in A$, $\hat{a}$ is a downward closed set in the tree $\{0,1\}^{<\mathbb{N}}$, and so  we consider $\hat{A}$ as a subspace of $\mathcal{P}(2^{<\mathbb{N}})$ which we can consider as $\{0,1\}^{2^{<{\mathbb{N}}}}$. In a similar way set $\hat{B}= \bigcup\limits_{b\in B}\hat{b}$ where $\hat{b}$ is the set of all infinite chains of $\{0,1\}^{< \mathbb{N}}$ whose union is equal to $b$. $\hat{A}$ and $\hat{B}$ are two orthogonal families of infinite subsets of $\{0,1\}^ {{<{\mathbb{N}}}}$. 
Let $Y=\hat {A}\times {\hat{B}}$. Note that the topology on $\hat{A}$ and $\hat{B}$ is induced by the topology on $H= {\{0,1\}}^{\{0,1\}^{<{\mathbb{N}}}}$, which is a copy of the Cantor space. Clearly, $H$ is compact and metrizable and so $\hat{A}$ and $\hat{B}$ are separable and metrizable. Therefore, $Y$ is separable and metrizable. Let $K_0$ be the set of all $\{(a,b),(a',b')\}$ from $[Y]^2$ such that $(a\cap b')\cup(b\cap a')\neq\emptyset$.  Since the product space $Y\times Y$ is Hausdorff, one  can find disjoint open subsets  about distinct pairs in $K_0$, and so $K_0$ is  an open subset of the product space $Y\times Y$ (see also {~\cite[p.\,146]{Fa}}). We shall now consider the open graph $G=(Y,K_0)$  and  apply OGA$^*(\Gamma$) to  $[Y]^{2}$. \\

%Recall that two famil ies $\mathcal A$ and $\mathcal B$ of subsets of  $\mathbb{N}$  are {\bf {countably separated}} if there are  sets $c_n$ $(n \in \mathbb{N})$ such that for every pair $a \in \mathcal{A},\ b \in \mathcal{B}$ there is an $n$ such that $a\subseteq^{*}c_{n}$ and $c_{n}\cap b$ is finite {~\cite[p.77]{Fa}}.  

If $A$ and $B$ are  complementary sets  in $E$, i.e., $ B=E\setminus A$, then one can consider $A$ and $B$ as  families of subsets of $\mathbb{N}$ such that $A\perp B$, and notice that $B^\perp \subseteq A$. By using a modification of  the result given in {~\cite[p.146]{Fa}} we have:

 \begin{lem}\label{farlem} $A$ is  countably generated in $B^\perp$ if and only if $G=(Y,K_0)$ is countably chromatic.
 \end{lem}

To see this, let $A$ be  countably generated in $B^\perp$. Then there is a sequence $\{c_{n}\}_{n\in\mathbb{N}}$ of elements of $B^\perp$ such that every element of $A$ is included, mod finite,  in one of the $c_n$'s. For $c\subseteq \mathbb{N}$, define $H_{c}=\{(a,b)\in Y: a\subseteq c$ and $b\cap c=\emptyset \}$. $H_{c}$ is  a subset of $Y$ and $[H_{c}]^2$ is disjoint from $K_{0}$.  Since every $a\in A$ is almost included  in one of the $c_n$'s,  define  $C_{n}=\bigcup\{c_{n}\cup s:  s\cap c_n=\emptyset,\, s\in A\cap [\mathbb{N}]^{< \mathbb{N}}\}$. Then  the union of  all  $H_{C_{n}}$'s covers $Y$. This implies $G$ is countably chromatic. For the other direction, let $Y$ be  decomposed into countably many subsets $Y_n$, i.e., $Y=\bigcup\limits_{n\in\mathbb{N}} Y_n$  where $[Y_n]^2\cap K_0=\emptyset$ for all $n\in\mathbb N$. Define ${c_{Y_n}}=\bigcup\limits_{(s,t)\in{Y_n}}s$, for all $Y_n$'s. Then every element of $A$ is almost included in one of  the $c_{Y_n}$'s. This implies $A$ is  countably generated in $B^\perp$.\qedhere

Since $A$ is not $\sigma$-compact, $A$ is  not countably generated in $B^\perp$.
To see this, suppose $A$ is  countably generated in $B^\perp$.  Then there is a sequence $\{c_n\}_{n\in\mathbb{N}} $ of elements of $B^\perp$, indeed elements of $A$, such that every element of $A$ is almost included in one of the $c_n$'s. For each $a\in A$, there is a $n\in  \mathbb{N}$ such that $a_{n}=a\setminus c_{n}\in[\mathbb{N}]^{< \mathbb{N}}$. Then $A=(\bigcup\limits_{n\in\mathbb{N}}{c_{n}})\cup(\bigcup\limits_{n\in\mathbb{N}} a_n)$. Notice that we identify the set of infinite branches of the tree $\{0,1\}^{< \mathbb{N}}$  with the infinite subsets of $\mathbb{N}$. Since each $c_{n}$ is a branch of the tree $\{0,1\}^{< \mathbb{N}}$, it is closed in $E$, and so it is compact. On the other hand, any finite subset of $E$ is compact. It follows that $A$ is $\sigma$-compact, a contradiction.\qedhere\\

% Since perfect subsets of separable metrizable spaces  have the cardinality of the continuum, set $F=\{(a_{\alpha},b_{\alpha})\in\hat{A}\times\hat{B}:  \alpha<2^{{\aleph}_{0}} \}$. 

Since  $G$ is not  countably chromatic, by $OGA^*(\Gamma)$, there is a perfect subset $F$ of $Y$ such that $[F]^{2}\subseteq K_{0}$. Since $F$ is perfect, there is a pruned tree $T_F$ such that $F=[T_F]$ {~\cite[Proposition 2.4]{Kech}}. Thus, an infinite $a \subseteq\mathbb{N}$ is in $p_{1}(F)$, where $p_{1}$ is the projection map onto the  first coordinate, if and only if  there is an infinite branch $f=\{(s_{k},t_{k})\}_{k\in  \mathbb{N}}$ of $T_F$ in  $\{0,1\}^{< \mathbb{N} }\otimes\{0,1\}^{< \mathbb{N}}$  such that $a$ is the union of  the $s_k$'s.  Since  $T_F $ is a pruned tree, it has a lexicographically least element, which is called  its \textbf{leftmost branch} {~\cite[p.\,9]{Kech}}. Let $f_{a}$ be the leftmost branch  of  $T_F$,  with the union of its first coordinates equal to $a$.\\

%One can find a leftmost branch $f_{ a}$ for each infinite $a\in p_{1}(F)$. Notice that $\hat{a}$ is the set of all infinite chains of $\{0,1\}^{< \mathbb{N}}$ with union equal to $a$. Indeed, $\hat a$ is the downward closure of $a$, so it is  a closed subset of $\{0,1\}^{< \mathbb{N}}$. By Zorn's lemma, there is a maximal infinite chain in $\hat{a}$. Therefore  there is an infinite branch of $\{0,1\}^{< \mathbb{N}}$ with union  equal to $a$. Moreover, $a$ is the projection of  a closed subset  $\hat{a}\times\hat{b}$ of $F$. Then $\hat{a}\times\hat{b}$ can be represented as the body of some pruned subtree of $\{0,1\}^{< \mathbb{N}}\otimes\{0,1\}^{< \mathbb{N}} $. Therefore, we can choose a leftmost branch for each infinite $a\in p_{1}(F)$].

All branches of $T_F$ are in the complete tree $\{0,1\}^{< \mathbb{N}}\otimes\{0,1\}^{< \mathbb{N}}$. By modifying Lemma 4.11 in \cite{jech} we have:

\begin{lem}If $F\subseteq{\{0,1\}}^{{\mathbb{N}}}\times{\{0,1\}}^{{\mathbb{N}}}$
is perfect , then\\ $T_{F}=\{(s,\tilde{s})\in
{\{0,1\}^{<{\mathbb{N}}}\otimes\{0,1\}^{<{\mathbb{N}}}\,:
(s,\tilde{s})\sqsubseteq (f,\tilde{f})\,{\text{for some}}\,
(f,\tilde{f})}\in F \}$ is a  perfect tree.
\end{lem}
\begin{proof}
Let $(s,\tilde{s})\in T_{F}$. Then there is a $(f,\tilde{f})\in F$
such that  $(s,\tilde{s})\sqsubseteq (f,\tilde{f})$. Indeed,  there is an $ n_{0}\in \mathbb{N}$ such that
$(f{\upharpoonright_{n_0}},\tilde{f}{\upharpoonright_{n_0}})=(s,\tilde{s})$
and \, $(s,\tilde{s})\sqsubseteq(f{\upharpoonright_{m}},\tilde{f}{\upharpoonright_{m}})$  for all
$m>n_0$. Since $F$ has no isolated point,  $(f,\tilde{f})$ is a limit point of $F$. Then there exists  a different point $(g,\tilde{g}) \in F$ such that   $(s,\tilde{s})\sqsubseteq (g,\tilde{g})$. Therefore, $(s,\tilde{s})$ has two incomparable extensions, that is $(s,\tilde{s})\sqsubseteq (g{\upharpoonright_{n_i}},\tilde{g}{\upharpoonright_{n_i}})$, $(s,\tilde{s})\sqsubseteq(f{\upharpoonright_{n_j}},\tilde{f}{\upharpoonright_{n_j}})$ for some $n_i,n_j\in \mathbb{N}$.\qedhere

\end{proof}

Now for $(s,t) \in T_F$,  define ${A}_{(s,t)}=\{ a\in p_{1}[F]: f_{a}$ extends $(s,t)\}$. Notice that ${A}_{(s,t)}$ is not countably generated in ${B}^\perp$. If it were, then $F$ would be a countable union of sets, call them $F_n$'s, such that $[ F_{n}]^2\cap K_0=\emptyset$ for all $n$. But this is a contradiction, since $[F]^{2}\subseteq K_0$.\\ 

To complete the proof  of $2.5$ we need a tree which consists of finite subsets of $\mathbb{N}$ as follows: let $B$ be a family of subsets of $\mathbb{N}$. $\Sigma\subseteq [\mathbb{N}]^{< \mathbb{N} }$ is called a \textbf{B-tree}  \cite{T2} if \\
 
 \noindent(i) $\emptyset\in \Sigma$,\\
 (ii) for every $\sigma\in \Sigma$, the set $\{ i\in \mathbb{N}: \sigma\cup \{i\}\in\Sigma\}$ is infinite and included in an element of $B$.\\

Now we will recursively construct a B-tree $\Sigma$ by using elements of $T_F$ and some elements of $B$ such that $\Sigma$ satisfies the  following conditions:\\
 
\noindent  1. If  $\sigma\subset \tau$ ($\tau$ strictly end-extends $\sigma$), then $(s_{\sigma}, t_{\sigma})\sqsubset (s_{\tau}, t_{\tau}) $ for $\tau,\sigma\in \Sigma$,\\
 2. $\sigma\subseteq s_{\sigma}$ for all $\sigma\in \Sigma$,\\
 3. $h_{\sigma}=\{i\in d_{\sigma}\in B: \sigma\cup\{i\}\in \Sigma\}$ is infinite for all $\sigma\in \Sigma$.\\

 Let $s_{\emptyset}= t_{\emptyset}=\emptyset$. Suppose we have some $\sigma\in \Sigma$ and  we know $(s_{\sigma},t_{\sigma})\in T_F$. Then  ${A}_{(s_{\sigma},t_{\sigma})}$ is not countably generated in ${B}^\perp$. Let $c_\sigma$ denote the union of  the elements of  $A_{(s_{\sigma},t_{\sigma})}$. Then $c_\sigma$ is not orthogonal to $B$. Therefore there is an element $d_\sigma$ of $B$ such that $c_{\sigma}\cap d_\sigma$ is infinite. Let $h_{\sigma}=\{i\in {c_{\sigma}\cap d_{\sigma}}: i > \mathrm{max(ran(s_{\sigma}))}\}$. For $i\in h_\sigma$ choose an element ${a}_{i}\in A_{(s_{\sigma},t_{\sigma})}$ such that $i \in {a}_i$ and denote by  $(s^{i},t^{i})$  the least element of the branch $f_{{a_i}}$ such that $ i\in s^i $.  Then put $\sigma\cup \{ i \}$ in $\Sigma$, and set $s_{\sigma \cup \{i\}}=s^i$ and $t_{\sigma \cup \{ i \}}=t^i$  for every such $i$. This completes the construction of $\Sigma$.\\

Recall that a tree $T$ is \textbf{superperfect} if for every $u\in T$ there is $v\in T$ extending $u$ such that $\{m\in \mathbb N: v\cup\{i\}\in T\}$ is infinite \cite{Kech77}. Notice that $\Sigma$ is a superperfect tree. It is known if $T$ is a superperfect tree, then there is a closed subset of $[T]$ which is homeomorphic to $\mathbb P$ {~\cite[p.\,163]{Kech}}. Observe that an infinite branch of  the B-tree $\Sigma$  enumerated increasingly is of the form $\{i_{0}, i_{1}, i_{2},\dots\}$ where $\sigma_{0}=\emptyset$, $\sigma_{1}=\{i_{0}\}$,  $\sigma_{2}=\{i_{0},i_{1}\}, \dots$. By construction of $\Sigma$, the pairs $(s_{{\sigma}_{i}},t_{{\sigma}_{i}})$, $i=1,2,\dots$, determine an infinite branch of $T_F$ whose projection is a member of $A$. Note also, by following item $2$ of the conditions on $\Sigma$, each $s_{{\sigma}_{i}}$ includes  the infinite branch $\{i_{0}, i_{1}, i_{2},\dots\}$. Then $[\Sigma]\subseteq A$.
Since $\Sigma$ is a superperfect tree,  there is a closed subset $Z$ of $A$, which is homeomorphic to $\mathbb P$, which was to be proved.\qedhere

\end{proof}

\begin{defn} The family $\mathcal{P}$ of projective sets of reals is  obtained by  closing the Borel sets under complementation and continuous image.
\end{defn}

\begin{thm}[~\cite{Fe}] If it is consistent that there is an inaccessible cardinal,  it is consistent that $OGA^*(\mathcal {P})$ holds, where $\mathcal{P}$ is the family of projective sets.

\end{thm}

\begin{cor} If it is consistent there is an inaccessible cardinal, it is consistent that every projective Menger set of reals is $\sigma$-compact.

\end{cor}

 \begin{proof}  The family of projective sets of reals is closed under continuous pre-images \cite{Kech}.\qedhere
\end{proof} 
Let us also note for use elsewhere that:
\begin{thm} $OGA^*$(co-analytic) implies every co-analytic Menger set of reals is $\sigma$-compact.
\end{thm}

 \begin{proof} 
 The family of co-analytic sets of reals is closed under continuous pre-images \cite{Kech}.\qedhere
 \end{proof}

\section{OGA$^*$(projective) and CH}
It turns out that {\bf{OGA}} itself is insufficient to imply Menger's conjecture for projective sets. This follows from the following two facts.
\begin{thm}[~ \cite{T2}] With ground model L, force OGA and MA  via a finite support iteration of length $\omega_2$  of ccc posets of size $\aleph_1$. Then $ \mathfrak c=\mathfrak b=\aleph_2$ and hence $\mathfrak d=\aleph_2$. In this model $\omega_1=\omega_{1}^L$. But 
\end{thm}

\begin{thm}[~ \cite{TTo}]\label{tt}  $\omega_1=\omega_{1}^{L[a]}$  for some $a\subseteq\omega$, and $\mathfrak d>\aleph_1$ imply there is a co-analytic Menger set of reals of size $\aleph_1$.
\end{thm}

\begin{proof} Such a set obviously cannot be $\sigma$-compact.
 The former hypothesis yields   ``co-analytic'' \cite{Kan94}, and the latter ``Menger'' \cite{Hur27}.
\end{proof}

\begin{thm}\label{14} If it is consistent that there is an inaccessible cardinal, it is consistent with CH that every Menger projective set of reals is $\sigma$-compact.

\end{thm}

\begin{proof} Let $V$ be a model of $OGA^*(projective)$ and let $V[G]$ be the result of collapsing $2^{\aleph_0}$ to $\aleph_1$ via countably closed forcing. Suppose in $V[G]$ there were a Menger projective non-$\sigma$-compact set of reals $P$. Note that there are no new reals in 
$V[G]$, and in fact no new open sets of reals. The first assertion is clear and well-known; for the second, the real line in the extension has the property that each open set is the union of countably many rational intervals, and there are no new such objects. Again, by countable closure, there are no new countable open covers. Since $P$ is projective, it is definable from a real $r$, which perforce is in $V$. The set of reals $r$ defines in $V$ is $P$, since $V[G]$ contains no new reals. Since $r$ codes a projective set,  $P$ is projective in $V$. If $P$ were $\sigma$-compact in $V$, say $P=\bigcup _{n\in\mathbb{N} } K_n$, then we claim $P$ would be $\sigma$-compact in $V[G]$. The reason is that $P$ is hereditarily Lindel\"of in $V[G]$, while countably closed forcing preserves countable compactness, so the compact witnesses for  a space's $\sigma$-compactness in $V$ would remain compact in $V[G$]. Finally, we claim $P$ is  also Menger and not $\sigma$-compact, contradicting $OGA^{*}(projective)$. If $P$ were Menger in $V[G]$, we claim $P$ would be Menger in $V$. Let $\{{\mathcal{U}}_n\}_{n\in\mathbb{N}}$ be a countable sequence of open covers of $P$ in $V$. Without loss of generality, we may assume each ${\mathcal{U}}_n$ is countable and a member of $V,$ and hence $\{{\mathcal U}_{n}\}_{n \in\mathbb{N }} \in V$. Let $\{\mathcal{V}_n\}_{n  \in\mathbb{N }}$ be a sequence of  finite subsets ${\mathcal{V}}_{n}\subseteq {\mathcal{U}}_{n}$ such that $\bigcup_{n\in\mathbb{N}} {\mathcal{V}}_n$ is a cover of $P$  in $V[G]$. Then each ${\mathcal{V}}_n$ and the sequence of ${\mathcal{V}}_n$'s are in $V$, and $\bigcup_{ n\in\mathbb{N}} {\mathcal{V}}_n$ covers $P$ there.\qedhere

\end{proof} 
 Another way of proving Theorem \ref{14} is to use Theorem 4.1 of \cite{Fe}, which asserts that $OGA^*(projective)$ holds in the model obtained by collapsing  an inaccessible to $\omega_1$
with finite conditions. CH can be arranged to hold in such a model, e.g. by assuming GCH in $V$; that shows $OGA^*(projective)$ does not imply $\mathbf{OGA}$.

\section{The inaccessible is necessary}
\begin{thm} \label{final}   The statement that every Menger co-analytic subset of $\mathbb R$ is $\sigma$-compact  is equiconsistent with the existence of an inaccessible cardinal.
\end{thm}

We actually have  a stronger version of the backward implication.  This follows from:
\begin{thm} [~ \cite{Fe}] The following are equivalent:
\item1.  $\omega_{1}^ { L[a] }<\omega_1$  for all $a\subseteq\omega$, 
\item 2. $OGA^*( \Sigma^1_2)$,
\item 3. $OGA^*( \Pi^1_1).$
\end{thm}

  The essential point is the observation of Specker (and G\"odel) (see for details {~\cite[11.6]{Kan94}) that if every co-analytic set of reals includes a perfect set, then $\omega_1$ is inaccessible in L. Feng \cite{Fe} mentions the following fact which he credits to the second author.  The direct implication in Theorem~\ref{final}  follows from this.

\begin{thm}\label{mengerinacc}
If  $\omega_{1}^{L[a]}=\omega_1$ for some $a \in\mathbb{N}^{ \mathbb{N}}$, then there is a co-analytic set of reals
which is not $\sigma$-compact but has the Menger property.
\end{thm}

\begin{proof}

  We shall rely on the following version of a standard fact (see \cite{Kan94}, page 171).

\begin{lem}\label{key} Assume $\omega_{1}^{L[a]}=\omega_1$ for some $a \in\mathbb{N}^{ \mathbb{N}}$. Then $\mathbb{N}^{ \mathbb{N}}\cap L[a]$ ordered by the relation $\leq^*$
of eventual dominance  has a co-analytic  $\omega_1$-scale, i.e., a  cofinal subset $A$ which is well-ordered by $ \leq^*$ in order type $\omega_1.$
\end{lem}

Note that such a set $A$ is not $\sigma$-compact and in fact not Borel.  This follows from the standard fact that a Borel well-founded relation on a Borel set of reals has countable rank (see {~\cite[p.239]{Kech}}). If $A$  has the Menger property, then the 
proof of Theorem \ref{mengerinacc} is finished. Otherwise, by a theorem of Hurewicz \cite{Hur28} there is a continuous
mapping $f: A\rightarrow \mathbb{N}^\mathbb{N}$ with range  cofinal in $(\mathbb{N}^\mathbb{N}, \leq^*).$ 
The map $f$ extends to a continuous map on a $G_\delta$-superset of $A$. Then  there is a Borel map
$g:\mathbb{N}^\mathbb{N}\rightarrow \mathbb{N}^\mathbb{N}$ such that $g\upharpoonright A=f$ (see  Theorem 12.2 in \cite{Kech}). Let $b\in \mathbb{N}^\mathbb{N}$ code both $a$ and the map $g.$ Then $\mathbb{N}^{ \mathbb{N}}\cap L[b]$ is cofinal in   $(\mathbb{N}^\mathbb{N}, \leq^*).$  Applying the previous  Lemma we obtain a co-analytic $\omega_1$-scale $B$ in $(\mathbb{N}^\mathbb{N},\leq^*).$  Adding to $B$ the countable set $D$ of all maps from $(\mathbb{N}\cup\{\infty\})^\mathbb{N}$
 that are eventually equal to $\infty$ we obtain an uncountable co-analytic set of reals $X$ concentrated around $D$. It is easily seen that this set has the Menger property.  However, such a set $X$ cannot be $\sigma$-compact for the  same reason why the set $B$ is not $\sigma$-compact (see \cite{TsZ} for a general result in this direction). This finishes the proof of Theorem \ref{mengerinacc}.\qedhere
 
 \end{proof}

We note that the same argument proves a stronger direct implication of Theorem \ref{mengerinacc}. A space  $X$ is {\it Hurewicz} if for any sequence $\{{\mathcal U}_n\}_{n \in\mathbb{N} }$ of open covers of $X$ one may pick finite
sets $\mathcal V_n\subseteq \mathcal U_n$ in such a way that $\{\bigcup\mathcal V_n:n\in\mathbb{N}\}$ is a  $\gamma$-cover of $X$. An  infinite open cover $\mathcal{U}$ is a {\it$\gamma$-cover} if for each $x\in X$ the set $\{U\in {\mathcal{U}}: x\not\in U\}$ is finite. Note that Hurewicz implies Menger but not conversely \cite{CP,Ts}.

\begin{thm} \label{finalhur}   The statement that every Hurewicz co-analytic subset of $\mathbb R$ is $\sigma$-compact  
implies that $\omega_1$ is an inaccessible cardinal in L.
\end{thm}
  
 As before it suffices to prove the following more precise statement. 

\begin{thm}\label{hurinacc}
If  $\omega_{1}^{L[a]}=\omega_1$ for some $a \in\mathbb{N}^{ \mathbb{N}}$ then there is a co-analytic set of reals
which is not $\sigma$-compact but has the Hurewicz property.
\end{thm}

\begin{proof}
We follow the proof of Theorem \ref{mengerinacc} very closely. Let $A$ be the co-analytic set given by Lemma \ref{key}. We know that  $A$ is not $\sigma$-compact and in fact not Borel.   If $A$  has the Hurewicz property then the 
proof of Theorem \ref{hurinacc} is finished. Otherwise, by another theorem of Hurewicz \cite{Hur25} there is a continuous
mapping $f: A\rightarrow \mathbb{N}^\mathbb{N}$ whose range is unbounded in $(\mathbb{N}^\mathbb{N}, \leq^*).$ 
The map $f$ extends to a continuous map on a $G_\delta$-superset of $A$. So   there is a Borel map
$g:\mathbb{N}^\mathbb{N}\rightarrow \mathbb{N}^\mathbb{N}$ such that $g\upharpoonright A=f.$ Let $b\in \mathbb{N}^\mathbb{N}$ code both $a$ and the map $g.$ Then $\mathbb{N}^{ \mathbb{N}}\cap\, L[b]$ is unbounded in   $(\mathbb{N}^\mathbb{N}, \leq^*).$  Applying  Lemma \ref{key} again, we obtain a co-analytic $\omega_1$-scale $B$ in $(\mathbb{N}^\mathbb{N},\leq^*).$  Adding to $B$ the countable set $D$ of all maps from $(\mathbb{N}\cup\{\infty\})^\mathbb{N}$
 that are eventually equal to $\infty$, we obtain an uncountable co-analytic set of reals $X$ concentrated around $D$. It is easily seen that concentrated sets have the Hurewicz property.  As before, such a set $X$ cannot be $\sigma$-compact,  so the proof  of Theorem \ref{hurinacc} is finished.
\end{proof}

In conclusion, let us thank the referee for many improvements.

{\rm Franklin D. Tall, Department of Mathematics, University of Toronto, \\Toronto, Ontario M5S 2E4, CANADA}\\
{\it e-mail address:} {\rm tall@math.utoronto.ca}\\

{\rm Stevo Todorcevic, Department of Mathematics, University of Toronto, \\Toronto, Ontario M5S 2E4, CANADA}\\
{\it e-mail address:} {\rm stevo@math.utoronto.ca}\\

{\rm Se{\c{c}}il Tokg\"oz, Department of Mathematics, Hacettepe University,\\ Beytepe, 06800, Ankara, TURKEY}\\
{\it e-mail address:} {\rm secilc@gmail.com}

\end{document}